\documentclass[oneside]{amsart}

\usepackage[utf8]{inputenc}
\usepackage{tikz}
\usepackage{tikz-cd}
\usepackage{amssymb}
\usepackage{hyperref}
\usepackage{enumerate}

\newtheorem{thm}{Theorem}[section]
\newtheorem{lem}[thm]{Lemma}
\newtheorem{prop}[thm]{Proposition}
\newtheorem{cor}[thm]{Corollary}

\newtheorem*{thma}{Theorem}

\theoremstyle{definition}
\newtheorem{defi}[thm]{Definition}
\newtheorem{ex}[thm]{Example}

\theoremstyle{remark}
\newtheorem{rem}[thm]{Remark}

\numberwithin{equation}{section}
\numberwithin{thm}{section}

\newcommand{\op}{\normalfont{^{op}}}

\newcommand{\Mod}[1]{\operatorname{mod} #1}

\newcommand{\Hom}[3]{\operatorname{Hom}_{#1}\left(#2,#3\right) }
\newcommand{\End}[2]{\operatorname{End}_{#1}\left(#2\right) }
\newcommand{\END}[2]{\operatorname{End}_{#1}(#2) }

\newcommand{\Soc}[1]{\operatorname{Soc}#1}
\newcommand{\Top}[1]{\operatorname{Top}#1}

\newcommand{\Rad}[1]{\operatorname{Rad} #1}
\newcommand{\RAD}[2]{\operatorname{Rad}^{#1} #2}

\newcommand{\Ima}[1]{\operatorname{Im}#1}
\newcommand{\Ker}[1]{\operatorname{Ker}#1}
\newcommand{\Coker}[1]{\operatorname{Coker}#1}

\newcommand{\Tr}[2]{\operatorname{Tr}\left(#1,#2\right)}
\newcommand{\TR}[2]{\operatorname{Tr}(#1,#2)}
\newcommand{\Rej}[2]{\operatorname{Rej}\left(#1,#2\right)}
\newcommand{\Ext}[4]{\operatorname{Ext}_{#1}^{#2}\left(#3,#4 \right)}

\newcommand{\B}[1]{\left(#1\right)}

\newcommand{\Gl}[1]{\operatorname{gl.dim}#1}

\newcommand{\Add}[1]{\operatorname{add}#1}

\newcommand{\LL}[1]{\operatorname{LL}(#1)}

\begin{document}

\title[The quasihereditary structure of the ADR algebra]{The quasihereditary structure of the Auslander-Dlab-Ringel algebra}
\author{Teresa Conde}
\address{Mathematical Institute,
University of Oxford, 
ROQ,
OX2 6GG, United Kingdom}
\email{\href{mailto:conde@maths.ox.ac.uk}{\nolinkurl{conde@maths.ox.ac.uk}}}
\thanks{The author would like to express her gratitude to the anonymous reviewer for the many valuable suggestions that led to substantial improvement of this paper. Namely, the author would like to thank the reviewer for noticing a mistake and providing a solution to it (Proposition \ref{prop:commoncf}), and for the simplified proof of Proposition \ref{prop:tiltinginj}. In addition, the author would like to thank her supervisor, Karin Erdmann, for her guidance and for the many useful suggestions in the preparation of this paper. This work was supported by FCT -- Funda\c{c}\~ao para a Ci\^encia e a Tecnologia, Portugal, through the grant SFRH/BD/84060/2012.}
\subjclass[2010]{Primary 16S50, 16W70. Secondary 16G10, 16G20.}
\keywords{Quasihereditary algebras, tilting modules, uniserial standard modules}
\date{May 1, 2016}

\begin{abstract}
Given an arbitrary algebra $A$ we may associate to it a special endomorphism algebra, $R_A$, introduced by Auslander. Dlab and Ringel constructed a heredity chain for $R_A$, proving that every algebra $A$ has an associated highest weight theory. In this paper we investigate the quasihereditary structure of $R_A$ using an axiomatic approach.
\end{abstract}

\maketitle

\section{Introduction}
\label{sec:intro}
Quasihereditary algebras were introduced in \cite{clineparshallscott} by Cline, Parshall and Scott, in order to deal with highest weight categories arising in the representation theory of Lie algebras and algebraic groups. This notion was extensively studied by Dlab and Ringel (\cite{MR987824}, \cite{MR943793}, \cite{MR1211481}, \cite[Appendix]{MR1284468}). Since the introduction of quasihereditary algebras, many classes of algebras arising naturally were shown to be quasihereditary.

A prototype for quasihereditary algebras are the Schur algebras, whose highest weight theory is that of general linear groups. They are the endomorphism algebras of certain modules over the group algebra of a symmetric group, and the algebra of the symmetric group can be seen as an idempotent subalgebra of the Schur algebra.

Thus it seemed natural that one can study an algebra $A$ by realising it as $(\xi R \xi, \xi)$ with $R$ quasihereditary and $\xi$ an idempotent in $R$. In \cite{MR0349747}, Auslander gave an explicit construction of an algebra $\tilde{R}_A$ and an idempotent $\xi \in \tilde{R}_A$ for every Artin algebra $A$, such that $\tilde{R}_A$ has finite global dimension, and $A$ is isomorphic to $(\xi \tilde{R}_A \xi, \xi)$. In \cite{MR943793}, Dlab and Ringel showed that this algebra $\tilde{R}_A$ is in fact quasihereditary. This may be rephrased by saying that any such $A$ has an associated highest weight theory.

In this paper, we study the basic algebra $R_A$ of $\tilde{R}_A$, where $A$ is a finite-di\-men\-sion\-al algebra over some field. We propose to call $R_A$ the \emph{Auslander-Dlab-Ringel algebra} (ADR algebra) of $A$. We show that $R_A$ satisfies the following two properties:
\begin{enumerate}
\item[(A1)] $\Rad{\Delta\B{i}}$ is either a standard module, or is zero;
\item[(A2)] if $\Rad{\Delta\B{i}}=0$ then the corresponding indecomposable injective module $Q_i$ has a filtration by standard modules (in other words, $Q_i$ is tilting).
\end{enumerate}
This motivates the following definition. Let $B$ be a quasihereditary algebra with respect to a poset $\B{\Phi, \sqsubseteq}$. We say that $B$ is \emph{ultra strongly quasihereditary} if it satisfies (A1) and (A2). This class of algebras is closed under Morita equivalence of quasihereditary algebras, since axioms (A1) and (A2) are expressed in terms of highest weight structures and of internal categorical constructions. By a result of Dlab and Ringel (\cite{yey}), condition (A1) implies that the category of modules with a $\Delta$-filtration is closed under submodules, and the algebras with this property were named ``strongly quasihereditary algebras" (\cite{RingelIyama}).

We prove several properties for algebras satisfying (A1) and (A2), and for their Ringel duals. In particular, we show that one can label the simple modules in a natural way by pairs $\B{i,j}$ so that $\Delta\B{i,j}$ has radical $\Delta\B{i,j+1}$ for $1 \leq j < l_i$ and $\Delta\B{i,l_i}$ is simple. As a main contribution of Section \ref{sec:costddinj}, we will prove the following (which corresponds to Theorem \ref{thm:qhstructA} and Proposition \ref{prop:tiltinginj}).
\begin{thma}
Let $B$ be an ultra strongly quasihereditary algebra. The injective hull $Q_{i,l_i}$ of the simple $B$-module with label $(i,l_i)$ has both a $\Delta$- and a $\nabla$-filtration. Moreover, the chain of inclusions
\[
0 \subset T\B{i,l_i} \subset \cdots \subset T\B{i,j} \subset \cdots \subset T\B{i,1}=Q_{i,l_i},
\]
where $T\B{i,j}$ is the tilting module corresponding to the label $\B{i,j}$, is the unique $\nabla$-filtration of $Q_{i,l_i}$. For $1\leq j < l_i$, the injective hull $Q_{i,j}$ of the simple module with label $(i,j)$ is isomorphic to $Q_{i,l_i}/T\B{i,j+1}$.
\end{thma}

The layout of the paper is the following. Section \ref{sec:prelim} contains background on quasihereditary algebras and on the ADR algebra. In Section \ref{sec:sect2}, we study the standard $R_A$-modules corresponding to the quasihereditary order $\B{\Lambda, \unlhd}$ of \cite{MR943793}. We prove that the uniserial projective $R_A$-modules described by Smal{\o} in \cite{smalo} are indeed standard modules with respect to $\B{\Lambda, \unlhd}$. In Section \ref{sec:qh}, we show that the algebra $R_A$ is quasihereditary with respect to $\B{\Lambda, \unlhd}$ -- our proof is different from that in \cite{MR943793}. Section \ref{sec:costddinj} introduces ultra strongly quasihereditary algebras. We prove the result on the labelling described previously, we construct the injective modules for these algebras and we prove Theorem \ref{thm:qhstructA}. Denote the Ringel dual of a quasihereditary algebra $B$ by $\mathcal{R}\B{B}$. In Section \ref{sec:rdual} we show that $\mathcal{R}\B{B} \op$ is ultra strongly quasihereditary whenever the algebra $B$ is an ultra strongly quasihereditary algebra. In Section \ref{sec:last} we determine a presentation of $R_A$ by quiver and relations when $A$ is a certain Brauer tree algebra, which occurs for example in the representation theory of the symmetric group.

\section{Preliminaries}
\label{sec:prelim}
Throughout this paper the word `algebra' will mean finite-dimensional $K$-algebra, where $K$ is some fixed field. Furthermore, all modules will be finite-dimensional left modules.

\subsection{The ADR algebra of \texorpdfstring{$A$}{[A]}}
\label{subsec:adralg}
Fix an algebra $A$. Given a module $M$, we shall denote its Loewy length by $\LL{M}$, that is, $\LL{M}$ is the minimal natural number such that $\RAD{\LL{M}}{M}=0$. Let $A$ have Loewy length $L$ (as a left module). We want to study the basic version of the endomorphism algebra of
\[
\bigoplus_{j=1}^L A/ \B{\Rad{A}}^j.
\]
This will have multiplicities in general.

Let $\{P_1, \ldots, P_n\}$ be a complete irredundant set of projective indecomposable $A$-modules and let $l_i$ be the Loewy length of $P_i$. Define
\[
G:=\bigoplus_{i=1}^{n} \bigoplus_{j=1}^{l_i} P_i/ \RAD{j}{P_i}.
\]
The modules $P_i/ \RAD{j}{P_i}$ are indecomposable and pairwise non-isomorphic, and these are precisely the indecomposable summands of $\bigoplus_{j=1}^L A/(\Rad{A})^j$ (up to isomorphism).

The algebra
\[
R=R_A:=\End{A}{G}\op ,
\]
which we call the \emph{ADR algebra of $A$}, is then a basic algebra of
\[\tilde{R}_A:=\END{A}{\bigoplus_{j=1}^L A/ \B{\Rad{A}}^j}\op .\]
The projective indecomposable $R$-modules are given by
\[
P_{i,j}:=\Hom{A}{G}{P_i/ \RAD{j}{P_i}},
\]
for $1 \leq i \leq n$, $1 \leq j \leq l_i$. Let $\xi \in R$ be the idempotent corresponding to the summand $\bigoplus_{i=1}^n P_{i,l_i}$ of $R$. Notice that $\xi R \xi$ is a basic algebra of $A$.

Denote the simple quotient of $P_{i,j}$ by $L_{i,j}$ and define
\[
\Lambda := \{ (i,j): 1 \leq i \leq n, \, 1 \leq j \leq l_i \},
\]
so that $\Lambda$ labels the simple $R$-modules.

The notation $\Mod{A}$ will be used for the category of (finite-dimensional) $A$-modules and, for every $M$ in $\Mod{A}$, $\Add{M}$ will denote the full subcategory of $\Mod{A}$ whose objects are the summands of finite direct sums of copies of $M$. We say that a set of modules (or a single module) $\Theta$ in $\Mod{A}$ \emph{generates} a module $M$, if $M$ is the image of some map $f$ whose domain is a (finite) direct sum of modules in $\Theta$. The notion of cogeneration is defined dually.

Since $G$ generates $A$, the functor $\Hom{A}{G}{-}$ has rather nice properties. Indeed, the functor
\[
\Hom{A}{G}{-}: \Mod{A} \longrightarrow \Mod{R}
\]
is fully faithful and it is right adjoint to the exact functor $\Hom{R}{\Hom{A}{G}{A}}{-}$. This implies that $\Hom{A}{G}{-}$ preserves injectives. Moreover, the restriction of $\Hom{A}{G}{-}$ to $\Add{G}$ yields an equivalence between the categories $\Add{G}$ and $\Add{R}$. A detailed account of the properties of this adjunction can be found in \cite[§$8$--§$10$]{MR0349747}. 

\subsection{Quasihereditary algebras}
\label{subsec:qhalg}
Given an algebra $B$ and a partial order $\B{\Phi, \sqsubseteq}$ labelling the simple $B$-modules, one defines the \emph{standard module} $\Delta(i)$, $i \in \Phi$, to be the largest quotient of $P_i$ with all composition factors of the form $L_j$, where $j \sqsubseteq i$. Here $L_i$ denotes the simple $B$-module with label $i \in \Phi$, and $P_i$ represents the projective $B$-module with top $L_i$. Let $Q_i$ be the injective $B$-module with socle $L_i$. The \emph{costandard module} $\nabla(i)$ is defined dually, by replacing `quotient' by `submodule', and $P_i$ by $Q_i$. The set of standard $B$-modules (resp.~costandard $B$-modules) is denoted by $\Delta$ (resp.~$\nabla$). Following \cite{MR1211481}, we say that the poset $\B{\Phi, \sqsubseteq}$ is \emph{adapted} to $B$ if the following holds: for every module $M$ with simple top $L_i$ and simple socle $L_j$, where $i$ and $j$ are incomparable in $\B{\Phi, \sqsubseteq}$, there is $k \in \Phi$ such that $k\sqsupset i$ or $k\sqsupset j$, and $[M:L_k]\neq 0$. Here $[M:L]$ denotes the Jordan-H\"{o}lder multiplicity of a simple module $L$ in $M$.

Any set of modules, $\Theta$, gives rise to the extension closed category $\mathcal{F}\B{\Theta}$ of all modules having a $\Theta$-filtration, i.e.~a filtration whose factors lie in $\Theta$ (up to isomorphism). The categories $\mathcal{F}\B{\Delta}$ and $\mathcal{F}\B{\nabla}$ are of central interest.

There are different equivalent ways of defining a quasihereditary algebra. We shall adopt the module theoretic perspective of \cite{MR1211481}. 
\begin{defi}
\label{defi:qh}
The algebra $B$ is \emph{quasihereditary} with respect to $\B{\Phi, \sqsubseteq}$ provided that:
\begin{enumerate}
\item $\B{\Phi, \sqsubseteq}$ is adapted to $B$;
\item the multiplicity of $L_i$ in $\Delta(i)$ is one for all $i \in \Phi$;
\item the projective modules lie in $\mathcal{F}\B{\Delta}$.
\end{enumerate}
\end{defi}
In this case we may write $(B, \Phi, \sqsubseteq)$. If $(B, \Phi, \sqsubseteq)$ is quasihereditary the dual of (3) also holds: the injective $B$-modules lie in $\mathcal{F}\B{\nabla}$.

Given a quasihereditary algebra $\B{B,\Phi, \sqsubseteq}$ and a module $M$ in $\mathcal{F}\B{\Delta}$, denote the multiplicity of $\Delta\B{i}$ in a $\Delta$-filtration of $M$ by $(M:\Delta\B{i})$. This number is independent of a choice of a $\Delta$-filtration, thus it is well defined. Quasihereditary algebras satisfy a Brauer–Humphreys type of reciprocity, which reduces to the identities $(P_i:\Delta\B{j})=[\nabla\B{j}:L_i]$ and $(Q_i:\nabla\B{j})=[\Delta\B{j}:L_i]$ when the field $K$ is algebraically closed (\cite[Lemma $2.5$]{MR1211481}). More generally, we have the following well-known result, which follows from \cite[Lemma $2.4$]{MR1211481}.
\begin{lem}
\label{lem:brauerhumph}
Let $\B{B, \Phi, \sqsubseteq}$ be a quasihereditary algebra. Let $M$ and $N$ be $B$-modules, with $M \in \mathcal{F}\B{\Delta}$ and $N \in \mathcal{F}\B{\nabla}$. Then, for $i \in \Phi$,
\begin{align*}
&(M:\Delta\B{i}) = \dim_{\End{B}{\nabla\B{i}}}{\Hom{B}{M}{\nabla\B{i}}}, \\
&(N:\nabla\B{i}) = \dim_{\End{B}{\Delta\B{i}}\op}{\Hom{B}{\Delta\B{i}}{N}}.
\end{align*}
\end{lem}
\section{The standard modules}
\label{sec:sect2}

Following the notation introduced in Subsection \ref{subsec:adralg}, recall that the set $\Lambda = \{ (i,j): 1 \leq i \leq n, \, 1 \leq j \leq l_i \}$ labels the simple modules over the ADR algebra $R$. Define a partial order, $\unlhd$, on $\Lambda$ by
\[
(i,j) \lhd (k,l) \Leftrightarrow j> l.
\]

We shall see, in Section \ref{sec:qh}, that the ADR algebra $R$ is quasihereditary with respect to $\B{\Lambda, \unlhd}$.
In this section, we describe the standard $R$-modules $\Delta\B{i,j}$ with respect to $\B{\Lambda, \unlhd}$. For this, two ingredients are needed. The following result, due to Smal{\o}, is crucial. 
 
\begin{prop}[{\cite[Proposition 2.1]{smalo}}]
\label{prop:smalo}
The modules $P_{1,1}, \ldots, P_{n,1}$ form a complete irredundant list of projective $R$-modules without proper projective submodules. Each projective $P_{i,1}$ is uniserial with Loewy length $l_i$ and, for every $(i,j)$ in $\Lambda$, we have the following short exact sequences
\[
\begin{tikzcd}[ampersand replacement=\&]
0 \arrow{r} \& \Hom{A}{G}{\Rad{P_i}/ \RAD{j}{P_i}} \arrow{r} \& P_{i,j} \arrow{r} \& \RAD{j-1}{P_{i,1}} \arrow{r} \& 0
\end{tikzcd}.
\]
\end{prop}
\begin{cor}
\label{cor:smaloconsequence}
For $1 \leq j \leq l_i$, the module $\RAD{j-1}{P_{i,1}}$ is uniserial and has composition factors $L_{i,j}, \cdots , L_{i,l_i}$, labelled from the top to the socle.
\end{cor}
\begin{proof}
By Proposition \ref{prop:smalo}, the projective indecomposable module $P_{i,1}$ has Loewy length $l_i$ and is uniserial. Thus, the module $\RAD{j-1}{P_{i,1}}$ is also uniserial and has Loewy length $l_i - j +1$. Note that $\RAD{k}{(\RAD{j-1}{P_{i,j}})}= \RAD{k+j-1}{P_{i,j}}$. By Proposition \ref{prop:smalo}, this module has a simple top isomorphic to $L_{i,k+j}$, for $0 \leq k \leq l_i-j$.
\end{proof}
The next lemma will also be used to determine the structure of the standard $R$-modules. Its proof can be found in \cite{MR0349747}, within the proof of Proposition~$10.2$.
\begin{lem}
\label{lem:verynice}
Let $M$ be in $\Mod{A}$. There is an epic $\varepsilon: X_0 \longrightarrow M$, with $X_0$ in $\Add{G}$ satisfying $\LL{X_0}= \LL{M}$, such that $\Hom{A}{G}{\varepsilon}$ is the projective cover of $\Hom{A}{G}{M}$ in $\Mod{R}$.
\end{lem}

Given a set of modules (or a single module) $\Theta$ and a module $M$ in $\Mod{A}$, define the \emph{trace of $\Theta$ in $M$}, $\Tr{\Theta}{M}$, to be the largest submodule of $M$ generated by $\Theta$ (see \cite[§8]{MR1245487}). If $B$ is an algebra endowed with a labelling poset $\B{\Phi, \sqsubseteq}$ (as in Subsection \ref{subsec:qhalg}), then $\Delta\B{i}=P_{i}/\TR{\bigoplus_{j: j \not\sqsubseteq i}P_j}{P_i}$ (see \cite[Lemma $1.1$]{MR1211481}).

\begin{prop}
\label{prop:standard}
The standard $R$-modules are uniserial. In fact,
\[
\Delta\B{i,j}\cong \RAD{j-1}{P_{i,1}},
\]
for every $(i,j)$ in $\Lambda$.
\end{prop}
\begin{proof}
By Proposition \ref{prop:smalo} and Corollary \ref{cor:smaloconsequence}, the module $\RAD{j-1}{P_{i,1}}$ is a quotient of $P_{i,j}$, and it has composition factors $L_{i,j}, \ldots, L_{i,l_i}$ (ordered from top to socle). So, by the definition of standard module, there must be an epic $f$ from $\Delta\B{i,j}$ to $\RAD{j-1}{P_{i,1}}$. Therefore we have the following commutative diagram
\[
\begin{tikzcd}[ampersand replacement=\&]
0 \arrow{r} \& \Tr{\bigoplus_{(k,l): (k,l)\not\unlhd (i,j)}P_{k,l}}{P_{i,j}} \arrow{r} \arrow[dashed, hook]{d}{\exists \, g} \& P_{i,j} \arrow{r} \arrow[equal]{d} \& \Delta\B{i,j} \arrow{r} \arrow[two heads]{d}{f} \& 0 \\
0 \arrow{r} \& \Hom{A}{G}{\Rad{P_i}/ \RAD{j}{P_i}} \arrow{r} \& P_{i,j} \arrow{r} \& \RAD{j-1}{P_{i,1}} \arrow{r} \& 0
\end{tikzcd}.
\]

Further, since $\LL{\Rad{P_i}/ \RAD{j}{P_i}}= j-1$, it follows from Lemma \ref{lem:verynice} that $\Hom{A}{G}{\Rad{P_i}/ \RAD{j}{P_i}}$ is generated by projectives $P_{k,l}$, such that $l<j$ (so $\B{k,l}\ntrianglelefteq \B{i,j}$). By the definition of trace, the inclusion map is an injection of $\Hom{A}{G}{\Rad{P_i}/ \RAD{j}{P_i}}$ into $\TR{\bigoplus_{(k,l): (k,l)\not\unlhd (i,j)}P_{k,l}}{P_{i,j}}$. Hence the composite of $g$ with this is one-to-one. But then the monic $g$ must be an isomorphism. Note that $\Ker{f} \cong \Coker{g}$, so the epic $f$ must be an isomorphism as well.\end{proof}

Observe that
\begin{equation}
\label{eq:radofstdd}
\Rad{\Delta\B{i,j}}=\Rad{\B{\RAD{j-1}{P_{i,1}}}}=
\begin{cases}
             \Delta\B{i,j+1}  & \text{if } j < l_i, \\
             0  & \text{if } j = l_i.
       \end{cases}
\end{equation}
Therefore $\Rad{\Delta\B{i,j}}$, which is the unique maximal submodule of $\Delta\B{i,j}$, belongs to $\mathcal{F}\B{\Delta}$ for all $(i,j)$ in $\Lambda$.

The next lemma can be found in \cite[Lemma $2$]{yey}. We state it for the convenience of the reader.
\begin{lem}
\label{lem:dlabringelfil}
Let $\Theta$ be a set of modules. Assume that for any $M$ in $\Theta$, every maximal submodule of $M$ has a $\Theta$-filtration. Then the category $\mathcal{F}\B{\Theta}$ is closed under submodules.
\end{lem}

By Lemma \ref{lem:dlabringelfil} and by the identity \eqref{eq:radofstdd}, the subcategory $\mathcal{F}\B{\Delta}$ of $\Mod{R}$ is closed under submodules. This suggests that there are many $R$-modules having a $\Delta$-filtration. In fact, the category $\mathcal{F}\B{\Delta}$ is at least as large as $\Mod{A}$. 
\begin{lem}
\label{lem:efedeltahoms}
Let $M$ be in $\Mod{A}$. The $R$-module $\Hom{A}{G}{M}$ belongs to $\mathcal{F}\B{\Delta}$.
\end{lem}
\begin{proof}
By Proposition \ref{prop:standard}, the result holds if $\LL{M}=1$. Assume the claim holds for modules with Loewy length $l-1$ and let $M$ have Loewy length $l$. The functor $\Hom{A}{G}{-}$ maps the short exact sequence 
\[
\begin{tikzcd}[ampersand replacement=\&]
0 \arrow{r} \& \Rad{M} \arrow{r} \& M \arrow{r} \& M/ \Rad{M} \arrow{r} \& 0
\end{tikzcd}
\]
to
\[
\begin{tikzcd}[ampersand replacement=\&, column sep=tiny]
0 \arrow{r} \& \Hom{A}{G}{\Rad{M}} \arrow{r} \& \Hom{A}{G}{M} \arrow{r} \arrow[two heads]{dr} \& \Hom{A}{G}{M/ \Rad{M}} \\
\& \& \& \Hom{A}{G}{M}/\Hom{A}{G}{\Rad{M}} \arrow[hook]{u}
\end{tikzcd}.
\]
By induction, $\Hom{A}{G}{\Rad{M}}$ lies in $\mathcal{F}\B{\Delta}$, and by the initial case, the module $\Hom{A}{G}{M/ \Rad{M}}$ belongs to $\mathcal{F}\B{\Delta}$ as well. According to Lemma \ref{lem:dlabringelfil}, $\mathcal{F}\B{\Delta}$ is closed under submodules, so
\[
\Hom{A}{G}{M}/\Hom{A}{G}{\Rad{M}} \in \mathcal{F}\B{\Delta}.
\]
The result follows from the fact that $\mathcal{F}\B{\Delta}$ is closed under extensions.
\end{proof}
\section{The ADR algebra is quasihereditary}
\label{sec:qh}

The ADR algebra is quasihereditary with respect to the heredity chain constructed by Dlab and Ringel in \cite{MR943793}. The underlying order in \cite{MR943793} can be shown to be the same as our partial order $\B{\Lambda, \unlhd}$. Instead of going into details about heredity chains, we give a different prove that $R$ is quasihereditary with respect to $\B{\Lambda, \unlhd}$.
\begin{lem}
\label{lem:adapted}
The partial order $\B{\Lambda, \unlhd}$ for the simple $R$-modules is an adapted order for $R$.
\end{lem}
\begin{proof}
Let $N$ be an indecomposable $R$-module. Suppose that $\Top{N}=L_{i,j}$ and $\Soc{N}=L_{k,l}$, with $(i,j)$ and $(k,l)$ incomparable with respect to $\unlhd$, i.e.~with $j=l$ and $i \neq k$. There is a nonzero morphism $f$ and a commutative diagram
\[
\begin{tikzcd}[ampersand replacement=\&]
\& P_{k,l} \arrow[dashed]{dl}[swap]{\exists \, t_*} \arrow{d}{f} \\
P_{i,l} \arrow[two heads]{r} \& N
\end{tikzcd}.
\]
Now $t_*=\Hom{A}{G}{t}$ for some $t:P_k/\RAD{l}{P_k} \longrightarrow P_i/\RAD{l}{P_i}$. The map $t$ must be a non-isomorphism since $k \neq i$. So $\Ima{t}$ is generated by a module in
\[
\mathcal{C}=\Add{\B{\bigoplus_{(x,y): y \leq l-1} P_x/\RAD{y}{P_x}}}.
\]
By the projectivity of $P_k/\RAD{l}{P_k}$ in $\Mod{(A/\B{\Rad{A}}^l)}$, we conclude that $t$ factors through a module in $\mathcal{C}$. Hence $t_*$ factors through a module in
\[
\Add{\B{\bigoplus_{(x,y): y \leq l-1} P_{x,y}}}.
\]
But then $N$ must have a composition factor of the form $L_{x,y}$ for some $x$ and some $y<l$, i.e.~for some pair $(x,y)$ such that $(x,y) \rhd (k,l)$.
\end{proof}
\begin{thm}
The algebra $R$ is quasihereditary with respect to $\B{\Lambda, \unlhd}$.
\end{thm}
\begin{proof}
We check that $\B{R, \Lambda, \unlhd}$ satisfies conditions (1) to (3) in Definition \ref{defi:qh}. By Lemma \ref{lem:adapted}, the poset $\B{\Lambda, \unlhd}$ is adapted to $R$. Proposition \ref{prop:standard} and Corollary \ref{cor:smaloconsequence} imply that $[\Delta\B{i,j}: L_{i,j}]=1$. Finally, recall that $P_{i,j}=\Hom{A}{G}{P_i/\RAD{j}{P_i}}$. By Lemma \ref{lem:efedeltahoms}, the projective indecomposable $R$-modules lie in $\mathcal{F}\B{\Delta}$.
\end{proof}

The next result, due to Dlab and Ringel (\cite{yey}, \cite[Lemma $4.1$*]{MR1211481}), is stated for completeness.
\begin{thm}
\label{thm:dlabringelfil}
Let $\B{B, \Phi, \sqsubseteq}$ be a quasihereditary algebra. The following assertions are equivalent:
\begin{enumerate}
\item $\Rad{\Delta\B{i}} \in \mathcal{F}\B{\Delta}$ for all $i \in \Phi$;
\item $\mathcal{F}\B{\Delta}$ is closed under submodules;
\item for all $i$ in $\Phi$ the module $\nabla\B{i}$ has injective dimension at most one;
\item every module in $\mathcal{F}\B{\nabla}$ has injective dimension at most one;
\item every torsionless module (i.e.~every module cogenerated by projectives) belongs to $\mathcal{F}\B{\Delta}$.
\end{enumerate}
\end{thm}
Consequently, assertions 1 to 4 hold for the quasihereditary structure of $R$, or, stated equivalently, $R$ is a \emph{right strongly quasihereditary algebra} (see \cite{RingelIyama}). Compare this statement with Observation (2) in \cite{RingelIyama} -- there the algebra $\Gamma$ is obtained by applying Iyama's construction to the regular module.

From now onwards denote the simple quotient of the $A$-module $P_i$ by $L_i$ and let $Q_i$ be the injective $A$-module with socle $L_i$. Similarly, let $Q_{i,j}$ be the injective $R$-module with socle $L_{i,j}$. We claim that the $R$-modules $Q_{i,l_i}$ have a $\Delta$-filtration.
\begin{lem}
\label{lem:niceinj}
The functor $\Hom{A}{G}{-}$ preserves indecomposable modules. In particular, $Q_{i,l_i}=\Hom{A}{G}{Q_i}$, and $\Hom{A}{G}{-}$ preserves injective hulls.
\end{lem}
\begin{proof}
The first assertion follows from the fact that $\Hom{A}{G}{-}$ is a fully faithful functor. Observe that $\Hom{A}{G}{-}$ also preserves injectives and note that the inclusion of $L_i$ in $Q_i$ induces a monic from $P_{i,1}$ (whose socle is $L_{i,l_i}$) to $\Hom{A}{G}{Q_i}$. So indeed $Q_{i,l_i}=\Hom{A}{G}{Q_i}$. Let now $M$ be in $\Mod{A}$ and suppose $\Soc{M}=\bigoplus_{j\in J}L_{x_j}$. Then $\bigoplus_{j\in J}P_{x_j,1}$ (whose socle is $\bigoplus_{j\in J}L_{x_j,l_{x_j}}$) is contained in $\Hom{A}{G}{M}$. Moreover, the functor $\Hom{A}{G}{-}$ maps the injective hull of $M$ to a monic from $\Hom{A}{G}{M}$ to $\bigoplus_{j\in J}Q_{x_j,l_{x_j}}$, so the statement follows.
\end{proof}

\section{Costandard, injectives and tilting modules}
\label{sec:costddinj}
Let $B$ be a quasihereditary algebra with respect to $(\Phi, \sqsubseteq)$. It was proved by Ringel in \cite{last} (see also Donkin, \cite{donkin}) that for every $i \in \Phi$ there is a unique indecomposable $B$-module $T\B{i}$ (up to isomorphism) which has both a $\Delta$- and a $\nabla$-filtration, with one composition factor labelled by $i$, and all the other composition factors labelled by $j$, $j\sqsubset i$. 

It is now standard to refer to a module in $\mathcal{F}\B{\Delta} \cap \mathcal{F}\B{\nabla}$ as a \emph{tilting module}. Let $T$ be the direct sum of the modules $T\B{i}$, $i \in \Phi$. This module is called the \emph{characteristic module} in \cite{last}, and it is such that $\Add{T}=\mathcal{F}\B{\Delta}\cap \mathcal{F}\B{\nabla}$.

Lemmas \ref{lem:efedeltahoms} and \ref{lem:niceinj} imply that the $R$-modules $Q_{i,l_i}$ belong to $\mathcal{F}\B{\Delta} \cap \mathcal{F}\B{\nabla} = \Add{T}$. Consequently, every module $Q_{i,l_i}$ is a direct summand of $T$.

In this section we:
\begin{enumerate}[(I)]
\item introduce the class of ultra strongly quasihereditary algebras, which contains the ADR algebras;
\item for $B$ an ultra strongly quasihereditary algebra, investigate the injective and the tilting modules -- our main results are Theorem \ref{thm:qhstructA} and Proposition~\ref{prop:tiltinginj}.
\end{enumerate}

So let $(B, \Phi , \sqsubseteq )$ be an arbitrary quasihereditary algebra, as before. Additionally, suppose that $B$ satisfies the following two conditions:
\begin{enumerate}
\item[(A1)] $\Rad{\Delta\B{i}} \in \Delta \cup \{0\}$ for all $i \in \Phi$;
\item[(A2)] $Q_i \in \mathcal{F}\B{\Delta}$ for all $i \in \Phi$ such that $\Rad{\Delta\B{i}}=0$.
\end{enumerate}
We call these algebras \emph{(right) ultra strongly quasihereditary} algebras. Note that the conditions in Theorem \ref{thm:dlabringelfil} hold for every ultra strongly quasihereditary algebra $\B{B, \Phi, \sqsubseteq}$. Moreover, the algebra $R_A$ is ultra strongly quasihereditary for every choice of $A$. However, notice that there are ultra strongly quasihereditary algebras which are not isomorphic to $R_A$ for any $A$.
\begin{ex}
\label{ex:example1}
Consider the path algebra $B=KQ$, where $Q$ is the quiver
\[
\begin{tikzcd}[ampersand replacement=\&]
  \overset{n}{\circ} \arrow{r} \&
    \overset{n-1}{\circ} \arrow{r} \&
      \cdots \arrow{r} \&
      \overset{1}{\circ}
\end{tikzcd}.
\]
The algebra $B$ is quasihereditary with respect to the natural ordering. Besides, $B$ satisfies (A1) and (A2). Yet $B$ is isomorphic to the quasihereditary algebra $R_A$ for some $A$ if and only if $n=1$.
\end{ex}

Let us start by stating some fundamental properties of the standard modules over an ultra strongly quasihereditary algebra.
\begin{lem}
\label{lem:littlelemma}
Let $\B{B, \Phi, \sqsubseteq }$ be an ultra strongly quasihereditary algebra. The standard $B$-modules are uniserial. Moreover, if $L_j$ is a composition factor of $\Delta\B{i}$, then $\Delta\B{j}$ is a submodule of $\Delta\B{i}$.
\end{lem}
\begin{proof}
The first part of the statement is a consequence of (A1). For the second part, as $L_j$ is a composition factor of $\Delta\B{i}$, there is a morphism $f:P_j \longrightarrow \Delta\B{i}$. So $\Ima{f}$ is a submodule of $\Delta\B{i}$ with simple top $L_j$. Therefore, we must have $\Ima{f} \cong \Delta\B{j}$.
\end{proof}

Given an ultra strongly quasihereditary algebra $\B{B, \Phi, \sqsubseteq}$, we may define a new order $\preceq$ on $\Phi$ by
\[
i\preceq j \Leftrightarrow \text{``}L_i\text{ is a composition factor of }\Delta\B{j}\text{"}.
\]
It follows from Lemma \ref{lem:littlelemma} that $\preceq$ is transitive and antisymmetric. Note that $\B{\Phi, \sqsubseteq}$ is a refinement of $\B{\Phi, \preceq}$, that is, $i\preceq j$ implies $i \sqsubseteq j$, $i, j \in \Phi$.
\begin{prop}
\label{prop:commoncf}

Let $\B{B, \Phi, \sqsubseteq }$ be an ultra strongly quasihereditary algebra. For each $i$ in $\Phi$, let $i^*$ be the element in $\Phi$ such that $\Soc{\Delta\B{i}}=L_{i^*}$. The following holds:
\begin{enumerate}
\item $L_{i^*}=\Delta\B{i^*}$ and $Q_{i^*} \in \mathcal{F}\B{\Delta}$;
\item if $i_1$ and $i_2$ are two maximal elements in $\B{\Phi, \preceq}$, and $\Delta\B{i_1}$ and $\Delta\B{i_2}$ have some composition factor in common, then $i_1=i_2$;
\item if $i$ is a maximal element in $\B{\Phi, \preceq}$ then $Q_{i^*}\cong T\B{i}$.
\end{enumerate}

\end{prop}
\begin{proof}

By Lemma \ref{lem:littlelemma}, every standard module $B$-module is uniserial. In particular, the modules $\Delta\B{i}$ have simple socle. For every $i \in \Phi$, write $i^*$ for the label in $\Phi$ such that $L_{i^*}=\Soc{\Delta\B{i}}$. Denote by ${\Phi}^*$ the set of all $i^*$.

Part (1) follows from Lemma \ref{lem:littlelemma} and from axiom (A2) in the definition of ultra strongly quasihereditary algebra.

For part (2) suppose, by contradiction, that $i_1$ and $i_2$ are two distinct maximal elements in $\B{\Phi, \preceq}$ such that the modules $\Delta\B{i_1}$ and $\Delta\B{i_2}$ have some common composition factor. Then, by Lemma \ref{lem:littlelemma}, we must have ${i_1}^*={i_2}^*=j$. By the injectivity of $Q_j$ and the uniseriality of $\Delta\B{i_1}$ and $\Delta\B{i_2}$, we get that the inclusion $L_j \longrightarrow Q_j$ can be extended to monomorphisms $\phi_x: \Delta\B{i_x} \longrightarrow Q_j$, $x=1,2$. As $i_1$ and $i_2$ are distinct and both maximal with respect to $\preceq$, then 
\begin{equation}
\label{eq:proofreviewer}
\Ima{\phi_1}\not\subseteq\Ima{\phi_2}, \,\,\, \Ima{\phi_2}\not\subseteq\Ima{\phi_1}.
\end{equation}
Now, by part (1), $Q_j$ lies in $\mathcal{F}\B{\Delta}$, i.e.~$Q_j$ has $\Delta$-filtration. Let $\Delta\B{k} \subseteq Q_j$ be such that $Q_j/ \Delta\B{k} \in \mathcal{F}\B{\Delta}$. Set $N:=Q_j/ \Delta\B{k}$. Since $\mathcal{F}\B{\Delta}$ is closed under submodules, $\Soc{N}$ must be a direct sum of simple modules $L_y$, with $y \in {\Phi}^*$. We cannot have simultaneously $\Ima{\phi_1} \subseteq \Delta\B{k}$ and $\Ima{\phi_2} \subseteq \Delta\B{k}$: by \eqref{eq:proofreviewer}, these two inclusions would produce two different composition series of $\Delta\B{k}$, which is impossible by Lemma \ref{lem:littlelemma}. So suppose, without loss of generality, that $\Ima{\phi_1} \not \subseteq \Delta\B{k}$. Then
\[
\Ima{\phi_1}/\B{\Ima{\phi_1} \cap \Delta\B{k}} \cong \B{\Ima{\phi_1}+ \Delta\B{k}}/\Delta\B{k}=:N'
\]
is a nonzero submodule of $N$. Since $L_j\subseteq \Ima{\phi_1} \cap \Delta\B{k}$ and $\Ima{\phi_1}\cong \Delta\B{i_1}$, Lemma \ref{lem:littlelemma} implies that every composition factor $L_y$ of $\Ima{\phi_1}/\B{\Ima{\phi_1} \cap \Delta\B{k}}$ is such that $y^*=j$, but $y\neq j$. In particular, $\Soc{N'}=L_{z}$, for some $z\not\in \Phi^*$. This is impossible since $\Soc{N'}\subseteq \Soc{N}$ and all the summands of $\Soc{N}$ are of the form $L_y$ with $y \in \Phi^*$. We get a contradiction. 

We concluded that for every $j \in \Phi^*$ there is exactly one maximal element $i$ in $\B{\Phi, \preceq}$ such that $i^*=j$. For part (3), consider the module $Q_{i^*}$, where $i$ is a maximal element in $\B{\Phi, \preceq}$. Note that $Q_{i^*}$ lies in $\mathcal{F}\B{\Delta} \cap \mathcal{F}\B{\nabla}=\Add{T}$: this follows from part (1) and from the fact that $B$ is a quasihereditary algebra. To conclude that $Q_{i^*}\cong T\B{i}$ it is enough to show that $[Q_{i^*}:L_i]\neq 0$ and that all composition factors of $Q_{i^*}$ are of the form $L_{x}$, with $x\sqsubseteq i$. Since $Q_{i^*}$ is the injective hull of $\Delta\B{i}$, we have $[Q_{i^*}: L_i]\neq 0$. By the Brauer–Humphreys reciprocity (Lemma \ref{lem:brauerhumph}), we get
\[
(Q_{i^*}: \nabla\B{y})=\dim_{\End{B}{\Delta\B{y}}\op}{\Hom{B}{\Delta\B{y}}{Q_{i^*}}}.
\]
So, for $(Q_{i^*}: \nabla\B{y})$ to be nonzero, we must have $y^*=i^*$, or equivalently, $y\preceq i$. Taking a $\nabla$-filtration of $Q_{i^*}$, we see that every composition factor $L_x$ of $Q_{i^*}$ must be a composition factor of some $\nabla\B{y}$ with $y\preceq i$. But for every composition factor $L_x$ of $\nabla\B{y}$ we have $x \sqsubseteq y$. Thus, for every composition factor $L_x$ of $Q_{i^*}$, there is $y$ such that $x\sqsubseteq y$ and $y\preceq i$. Therefore, $x\sqsubseteq i$.
\end{proof}

Let $\B{B, \Phi, \sqsubseteq}$ be an ultra strongly quasihereditary algebra. Suppose $i$ is maximal with respect to $\B{\Phi,\preceq}$. The module $\Delta\B{i}$ is uniserial. Assume $\Delta\B{i}$ has Loewy length $l_i$ and, by analogy with $R$, let $L_{i_1}, \ldots, L_{i_{l_i}}$ be the composition factors of $\Delta\B{i}$, ordered from the top to the socle (so $i_1=i$ and $i_{l_i}=i^*$). We may relabel the simple $B$-modules as $(i,j)$, where, for every maximal $i$ in $\B{\Phi,\preceq}$, the label $i$ is replaced by $(i,1)$, and the remaining labels $i_j$ (as before) are replaced by $(i,j)$. By the definition of the partial order $\B{\Phi,\preceq}$, every simple $B$-module has been given such a label. Furthermore, Proposition \ref{prop:commoncf} assures that this relabelling is well defined. Note that this relabelling is consistent with the labels chosen for the simple $R$-modules. From now onwards we will use this new labelling for the simple $B$-modules. I.e., we shall assume (unless otherwise stated) that $(B,\Phi,\sqsubseteq)$ denotes an ultra strongly quasihereditary algebra and that
\[
\Phi = \{(i,j): 1 \leq i \leq n , 1 \leq j \leq l_i \}
.\]
So $L_{i,j}$, $P_{i,j}$, $Q_{i,j}$, $\Delta\B{i,j}$, $\nabla\B{i,j}$, $T\B{i,j}$ and $T$ will be the naturally expected $B$-modules.

Consider an injective $B$-module of type $Q_{i,l_i}$. By Proposition \ref{prop:commoncf}, $Q_{i,l_i}$ is isomorphic to $T(i,1)$. As we shall see shortly, every $T\B{i,j}$ may be determined recursively from $T\B{i,1}$. The next lemma will be useful when proving this claim.

\begin{lem}
\label{lem:newone}
Let $\B{B, \Phi, \sqsubseteq}$ be an arbitrary quasihereditary algebra. For $i \in \Phi$ consider the short exact sequence
\begin{equation}
\label{eq:sescostandard}
\begin{tikzcd}[ampersand replacement=\&]
0 \arrow{r} \& Y\B{i} \arrow{r} \& T\B{i} \arrow{r}{\psi} \& \nabla\B{i} \arrow{r} \& 0
\end{tikzcd},
\end{equation}
as in \cite[Section 5]{last} (i.e.~with $\psi$ a right minimal $\mathcal{F}\B{\Delta}$-approximation of $\nabla\B{i}$ and with $Y\B{i}$ a module lying in $\mathcal{F}\B{\{\nabla \B{j} : j \sqsubset i\}}$). Then:
\begin{enumerate}
\item $\Rad{\Delta\B{i}}$ is a submodule of $Y\B{i}$;
\item for every morphism $f: T\B{i} \longrightarrow \nabla \B{i}$, there is a map $h$ in the division algebra $\End{B}{\nabla\B{i}}$ such that $f= h \circ \psi$;
\item if $M \subseteq T\B{i}$, with $M$ in $\mathcal{F}\B{\nabla}$ and $T\B{i}/ M$ a costandard module, then $T\B{i}/ M = \nabla\B{i}$ and $M=Y\B{i}$.
\end{enumerate}
\end{lem}
\begin{proof}
There is an exact sequence
\begin{equation}
\label{eq:sesstandard}
\begin{tikzcd}[ampersand replacement=\&]
0 \arrow{r} \& \Delta\B{i} \arrow{r}{\phi} \& T\B{i} \arrow{r} \& X\B{i} \arrow{r} \& 0
\end{tikzcd},
\end{equation}
dual to \eqref{eq:sescostandard} (see \cite[Section 5]{last}), where $X\B{i}$ lies $\mathcal{F}\B{\{\Delta \B{j} : j \sqsubset i\}}$. So we may regard $\Delta\B{i}$ as a submodule of $T\B{i}$. The image of $\Delta\B{i}$ under $\psi$ must be the socle of $\nabla\B{i}$, since $L_i$ occurs only once as a composition factor of $T\B{i}$. This proves part (1).

Now apply the functor $\Hom{B}{-}{\nabla\B{i}}$ to \eqref{eq:sesstandard}. We have $\Hom{B}{X\B{i}}{\nabla \B{i}}=0$, as $L_i$ is not a composition factor of $X\B{i}$. Because of this, and also because $\Ext{B}{1}{\mathcal{F}\B{\Delta}}{\mathcal{F}\B{\nabla}}=0$ (see \cite[Theorem $1$]{MR1211481}), we get an isomorphism
\[
\Hom{B}{T \B{i}}{\nabla \B{i}} \longrightarrow \Hom{B}{\Delta\B{i}}{\nabla \B{i}}
\]
of $S$-modules, where $S:=\End{B}{\nabla\B{i}}$ is a division algebra. As $\Hom{B}{\Delta\B{i}}{\nabla \B{i}}$ is 1-dimensional over $S$, part (2) follows.

For part (3), note that the epic $f:T\B{i} \longrightarrow T\B{i}/M$ must be a right $\mathcal{F}\B{\Delta}$-approximation of $T\B{i}/ M$, as $\Ext{B}{1}{\mathcal{F}\B{\Delta}}{M}=0$ (consult \cite[pages 113, 114]{ausrei} for the definition of right approximation). Since $T\B{i}$ is an indecomposable module, the map $f$ is indeed a right minimal $\mathcal{F}\B{\Delta}$-approximation of $T\B{i}/ M$ (see \cite[Proposition $1.1$, (a)]{ausrei}). Suppose $T\B{i}/M = \nabla \B{j}$. So both $f$ and $\psi: T\B{j} \longrightarrow \nabla\B{j}$ are right minimal $\mathcal{F}\B{\Delta}$-approximations of $\nabla\B{j}$. As a consequence, $T\B{j}$ and $T\B{i}$ must be isomorphic (see \cite[page 114]{ausrei}), so $j=i$. If we look at $Y\B{i}$ as a submodule of $T\B{i}$, then part (2) implies that $\iota=\iota' \circ t$, where $t$ is an isomorphism and $\iota:Y\B{i} \longrightarrow T\B{i}$, $\iota ':M \longrightarrow T\B{i}$ are the inclusion maps. Thus $M=Y\B{i}$.
\end{proof}

We are now in position of proving one of our main results.
\begin{thm}
\label{thm:qhstructA}
Let $(B, \Phi, \sqsubseteq)$ be an ultra strongly quasihereditary algebra. Then $Q_{i,l_i}=T\B{i,1}$ and, for every $(i,j) \in \Phi$, we have the following short exact sequence
\begin{equation}
\label{eq:tiltingses}
\begin{tikzcd}[ampersand replacement=\&]
0 \arrow{r} \& T\B{i,j+1} \arrow{r} \& T\B{i,j} \arrow{r}{\psi} \& \nabla\B{i,j} \arrow{r} \& 0
\end{tikzcd},
\end{equation}
where $T\B{i,l_i +1}:=0$. In particular,
\begin{equation}
\label{eq:filtrationT}
0 \subset T\B{i,l_i} \subset \cdots \subset T\B{i,j} \subset \cdots \subset T\B{i,1}=Q_{i,l_i}
\end{equation}
is the unique $\nabla$-filtration of $T\B{i,1}$.
\end{thm}
\begin{proof}
By Proposition \ref{prop:commoncf}, we must have $Q_{i,l_i}=T\B{i,1}$. We will prove by induction on $k$, that there is a filtration 
\[
T\B{i, k} \subset T\B{i, k-1} \subset \cdots \subset T\B{i,1}=Q_{i,l_i}.
\]
For $k=1$ the claim is obvious. Suppose the claim holds for all $k\leq j$. So assume that $T\B{i,j}\subseteq T\B{i,1}$, and consider the short exact sequence
\[
\begin{tikzcd}[ampersand replacement=\&]
0 \arrow{r} \& Y\B{i,j} \arrow{r} \& T\B{i,j} \arrow{r}{\psi} \& \nabla\B{i,j} \arrow{r} \& 0
\end{tikzcd}
\]
(as in \eqref{eq:sescostandard}). Suppose $j \neq l_i$. Then $\psi$ cannot be an isomorphism, as $\nabla\B{i,j} $ is not in $ \mathcal{F}\B{\Delta}$. Since $Y\B{i,j}\subseteq Q_{i,l_i}$ and $\Soc{Q_{i,l_i}}=L_{i,l_i}$ is simple, we get that $\Soc{Y\B{i,j}}=L_{i,l_i}$. Therefore $Y\B{i,j}$ is indecomposable. As $\mathcal{F}\B{\Delta}$ is closed under submodules (recall Theorem \ref{thm:dlabringelfil}), we must have $Y\B{i,j} \in \mathcal{F}\B{\Delta} \cap \mathcal{F}\B{\nabla}$. Thus $Y\B{i,j}=T\B{i,l}$, for some $1 \leq l \leq l_i$ (note that $\Delta\B{k,l} \subseteq T\B{k,l}$, so $T\B{k,l}$ must have the summand $L_{k,l_k}$ in its socle). From Lemma \ref{lem:newone}, we also know that $\Rad{\Delta\B{i,j}}=\Delta\B{i,j+1}$ is contained in $Y\B{i,j}$. Hence $(i,j+1) \sqsubseteq (i,l)$, so $j+1 \geq l$. We cannot have $l \leq j$, otherwise, as $\Delta\B{i,l}$ is a submodule of $Y\B{i,j}$, $L_{i,j}$ would be a composition factor of $Y\B{i,j}$. Thus $l=j+1$ and $Y\B{i,j}=T\B{i,j+1}$. Note that $Y\B{i,l_i}=0$, otherwise $Y\B{i,l_i}$ would have socle $L_{i,l_i}$. Therefore we get a $\nabla$-filtration as in \eqref{eq:filtrationT}, and part (3) of Lemma \ref{lem:newone} assures its uniqueness.
\end{proof}
\begin{rem}
\label{rem:factor}
Let $1 \leq j < j' \leq l_i$. Then $T\B{i,j'}$ is a submodule of $T\B{i,j}$. We assert that $T\B{i,j}/ T\B{i,j'}$ must be an indecomposable $R$-module. First, note that $T\B{i,j}/ T\B{i,j'}$ belongs to $\mathcal{F}\B{\nabla}$. Indeed, this module must have a unique $\nabla$-filtration as this is the case of $T\B{i,1}$ (look at \eqref{eq:filtrationT}). Since $\mathcal{F}\B{\nabla}$ is closed under direct summands, every module having a unique $\nabla$-filtration must be indecomposable. 
\end{rem}

Given a set of modules (or a single module) $\Theta$ and a module $M$ in $\Mod{A}$, define the \emph{reject of $\Theta$ in $M$}, $\Rej{M}{\Theta}$, to be the submodule $N$ of $M$ such that $M/N$ is the largest factor module of $M$ cogenerated by $\Theta$ (see \cite[§8]{MR1245487}). From the filtration \eqref{eq:filtrationT} and by the properties of $\nabla$-filtrations it is not difficult to conclude that
\[
T\B{i,j}=\Rej{Q_{i,l_i}}{\bigoplus_{(k,l):(k,l)\sqsupset (i,j)}Q_{k,l}}=\Rej{Q_{i,l_i}}{\bigoplus_{(k,l):(k,l)\not\sqsubseteq (i,j)}Q_{k,l}}
\]
Therefore, we have the following result.
\begin{lem}
\label{lem:rejfiltration}
Let $(B, \Phi, \sqsubseteq)$ be an ultra strongly quasihereditary algebra. The module $T\B{i,j}$, $(i,j) \in \Phi$, is the largest submodule of $Q_{i,l_i}$ whose all composition factors are of the form $L_{k,l}$, $(k,l)\sqsubseteq (i,j)$.
\end{lem}

We now claim that $Q_{i,j}/ \nabla\B{i,j}$ is isomorphic to $Q_{i,j-1}$ for $1 < j \leq l_i$, and that $Q_{i,1}\cong \nabla\B{i,1}$.
\begin{prop}
\label{prop:tiltinginj}
Let $(B, \Phi, \sqsubseteq )$ be an ultra strongly quasihereditary algebra. For every $(i,j) \in \Phi$, we have the short exact sequences
\begin{equation}
\label{eq:sesinjultrastrong}
\begin{tikzcd}[ampersand replacement=\&]
0 \arrow{r} \& \nabla\B{i,j} \arrow{r} \& Q_{i,j} \arrow{r} \& Q_{i,j-1} \arrow{r} \& 0,
\end{tikzcd}
\end{equation}
\begin{equation}
\label{eq:sesinjultrastrong2}
\begin{tikzcd}[ampersand replacement=\&]
0 \arrow{r} \& T\B{i,j+1} \arrow{r} \& T\B{i,1} \arrow{r} \& Q_{i,j} \arrow{r} \& 0,
\end{tikzcd}
\end{equation}
where $Q_{i,0}:=0$. Moreover, the module $Q_{i,j}$ has a unique $\nabla$-filtration.
\end{prop}
\begin{proof}
By Theorem \ref{thm:qhstructA}, we have the exact sequences
\begin{equation}
\label{eq:sesinj2}
\begin{tikzcd}[ampersand replacement=\&, column sep=small]
 0 \arrow{r} \& T\B{i,j}/T\B{i,j+1} \arrow{r} \& Q_{i,l_i}/T\B{i,j+1} \arrow{r} \& Q_{i,l_i}/T\B{i,j} \arrow{r} \& 0
\end{tikzcd},
\end{equation}
where $T\B{i,l_i+1}=0$ and $T\B{i,j}/T\B{i,j+1} \cong \nabla\B{i,j}$, $1 \leq j \leq l_i$. By Theorem \ref{thm:dlabringelfil}, the modules $T\B{i,j}$, $1 \leq j \leq l_i$, have injective dimension at most one. As $Q_{i,l_i}$ is the injective hull of $T\B{i,j}$, we get that all $Q_{i,l_i}/T\B{i,j}$ are injective. The modules $Q_{i,l_i}/T\B{i,j+1}$ have a unique $\nabla$-filtration by Theorem \ref{thm:qhstructA}, so they are indecomposable (see Remark \ref{rem:factor}). Therefore $Q_{i,l_i}/T\B{i,j+1}$ is the injective hull of $\nabla\B{i,j}$ for every $1 \leq j \leq l_i$, which shows that $Q_{i,l_i}/T\B{i,j+1} = T\B{i,1}/T\B{i,j+1}$ is isomorphic to $Q_{i,j}$. This produces the short exact sequence \eqref{eq:sesinjultrastrong2} in the statement of this proposition. Now \eqref{eq:sesinj2} gives the exact sequence \eqref{eq:sesinjultrastrong}.
\end{proof}

\section{The Ringel dual}
\label{sec:rdual}
In this section we start by summarising the general setup for the Ringel dual of a quasihereditary algebra. Then, we study the Ringel dual $\mathcal{R}\B{B}$ of an ultra strongly quasihereditary algebra $B$. The main goal of this section is to show that $\mathcal{R}\B{B}\op$ is also ultra strongly quasihereditary.

For now suppose that $(B, \Phi, \sqsubseteq )$ is an arbitrary quasihereditary algebra. Denote by $L_i$, $Q_i$, $\nabla\B{i}$, $T\B{i}$ and $T$, respectively, the simple $B$-modules, the injective indecomposables, etc., as naturally expected. The algebra $\END{B}{T}\op$ is quasihereditary with respect to the poset $(\Phi , \sqsubseteq \op )$. This endomorphism algebra, investigated by Ringel in \cite{last}, is known as the \emph{Ringel dual} of $B$, and we shall denote it by $\mathcal{R}\B{B}$. It was shown in \cite{last} that $\mathcal{R}\B{\mathcal{R}\B{B}}\cong B$, for $B$ basic. 

Denote by $P_i'$ the projective indecomposable $\mathcal{R}\B{B}$-module $\Hom{B}{T}{T\B{i}}$ and let $L_i'$ be its top. Denote the standard, the costandard and the summands of the characteristic $\mathcal{R}\B{B}$-module $T'$ accordingly (with the prime symbol).

The restriction of the functor
\[
\Hom{B}{T}{-}: \Mod{B} \longrightarrow \Mod{\mathcal{R}\B{B}}
\]
to $\mathcal{F}\B{\nabla}$ yields an equivalence between the categories $\mathcal{F}\B{\nabla}$ and $\mathcal{F}\B{\Delta '}$. 

Since $\Ext{B}{1}{T}{\mathcal{F}\B{\nabla}}$ vanishes, then $\Hom{B}{T}{-}$ maps short exact sequences in $\Mod{B}$ with modules in $\mathcal{F}\B{\nabla}$ to short exact sequences in $\Mod{\mathcal{R}\B{B}}$ with modules in $\mathcal{F}\B{\Delta '}$.

The following holds
\begin{equation*}
\begin{array}{l}
 \Hom{B}{T}{T\B{i}}=P_i', \\
 \Hom{B}{T}{\nabla\B{i}}=\Delta '\B{i}, \\
 \Hom{B}{T}{Q_i}=T'\B{i}.
\end{array}
\end{equation*}

\subsection{Ringel dual of an ultra strongly quasihereditary algebra}
Now we assume that $(B, \Phi, \sqsubseteq )$ is an ultra strongly quasihereditary algebra and label the simple $B$-modules by $(i,j)$, as described in Section \ref{sec:costddinj}. We want to show that $\mathcal{R}\B{B}\op$ is ultra strongly quasihereditary.

Let $D$ be the standard duality. Then the standard modules over $\mathcal{R}\B{B}\op$ are the modules $D(\nabla ' \B{i,j})$, and the indecomposable injectives are the modules $D(P_{i,j}')$. To verify that (A1) and (A2) hold for $\mathcal{R}\B{B}\op$, we need that
\begin{enumerate}
\item[(A1*)] $\nabla ' \B{i,j}/ L_{i,j}'$ is either a costandard module, or is zero;
\item[(A2*)] if $\nabla ' \B{i,j}$ is simple, then $P_{i,j}'$ has a $\nabla '$-filtration (that is, it is a tilting module).
\end{enumerate}

From the quasihereditary struture of $B$ we can immediately deduce some properties of $\mathcal{R}\B{B}$.
\begin{enumerate}[(I)]
\item We have that $P_{i,1}'\cong T'\B{i,l_i}$ since $T\B{i,1}$ is isomorphic to $Q_{i,l_i}$.
\item By applying the functor $\Hom{B}{T}{-}$ to the exact sequence \eqref{eq:tiltingses} in the statement of Theorem \ref{thm:qhstructA}, we get
\[
\begin{tikzcd}[ampersand replacement=\&]
0 \arrow{r} \& P'_{i,j+1} \arrow{r} \& P'_{i,j} \arrow{r} \& \Delta'\B{i,j} \arrow{r} \& 0
\end{tikzcd},
\]
where $P'_{i,l_i+1}:=0$. In particular, the standard $\mathcal{R}\B{B}$-modules have projective dimension at most one. By \cite[Lemma $4.1$]{MR1211481}, this is equivalent to the fact that $\mathcal{F}\B{\nabla '}$ is closed under factor modules.
\item Using the functor $\Hom{B}{T}{-}$ we get from the filtration \eqref{eq:filtrationT} that the module $P_{i,1}'\cong T'\B{i,l_i}$ has a unique $\Delta ' $-filtration, given by
\[
0 \subset P'_{i,l_i} \subset \cdots \subset P'_{i,j} \subset \cdots \subset P'_{i,1}=T'\B{i,l_i}.
\]
The quotients are as described in (II).
\end{enumerate}

\begin{thm}\label{thm:lastonebynow}
Using the notation introduced previously, we have:
\begin{enumerate}
\item $P'_{i,1} \cong T'\B{i,l_i}$;
\item if $1 \leq j < l_i$, then $T'\B{i,j}\cong P'_{i,1}/ P'_{i,j+1}$;
\item for $(i,j) \in \Phi$, the costandard module $\nabla ' \B{i,j}$ has Loewy length $j$, is uniserial, and satisfies 
\[
\nabla ' \B{i,j-1} \cong \nabla ' \B{i,j} / L_{i,j}'.
\]
\end{enumerate}
\end{thm}
\begin{proof}
Part (1) is answered in (I) above. Part (2) follows by applying the functor $\Hom{B}{T}{-}$ to \eqref{eq:sesinjultrastrong2} in Proposition \ref{prop:tiltinginj}.

To prove part (3) apply Lemma \ref{lem:brauerhumph} to (III). This yields
\begin{multline}
\label{eq:multline}
\dim_{\End{\mathcal{R}\B{B}}{\nabla'\B{i,j}}}{\Hom{\mathcal{R}\B{B}}{P'_{k,l}}{\nabla'\B{i,j}}} \\
=(P'_{k,l}:\Delta '\B{i,j})=
\begin{cases}
1 & \text{if }k=i\text{ and }l\leq j, \\
0 & \text{otherwise}.
\end{cases}
\end{multline}
As a consequence, the composition factors of $\nabla ' \B{i,j}$ are $L'_{i,1}, \ldots, L'_{i,j}$, with $L'_{i,j}$ having multiplicity one in $\nabla ' \B{i,j}$. In particular, $\nabla'\B{i,1}\cong L'_{i,1}$. We prove that $\nabla ' \B{i,j}/L_{i,j}' \cong \nabla '\B{i,j-1}$ for $1 < j \leq l_i$. Let $L$ be a direct summand of $\Top{\nabla '\B{i,j}}$. Since $\mathcal{F}\B{\nabla ' }$ is closed under taking quotients, then $L$ must be a costandard module. By \eqref{eq:multline}, we must have $L \cong \nabla ' \B{i,1} \cong L'_{i,1}$. Thus, there is an exact sequence
\[
\begin{tikzcd}[ampersand replacement=\&]
0 \arrow{r} \& \Ker{\pi} \arrow{r}{\iota} \& \nabla'\B{i,j} \arrow{r}{\pi} \& \nabla'\B{i,1} \arrow{r} \& 0
\end{tikzcd}.
\]
We claim that $[\nabla ' \B{i,j}: L'_{i,1}] =1$. For this, let $M$ be a submodule of $\Ker{\pi}$ generated by $P'_{i,1}=T'\B{i,l_i}$. Since $\mathcal{F}\B{\nabla '}$ is closed under quotients, it follows that $M \in \mathcal{F}\B{\nabla '}$, but also $\nabla'\B{i,j}/ M \in \mathcal{F}\B{\nabla '}$, i.e.~there is an exact sequence
\[
\begin{tikzcd}[ampersand replacement=\&]
0 \arrow{r} \& M \arrow{r} \& \nabla'\B{i,j} \arrow{r} \& \nabla'\B{i,j}/ M \arrow{r} \& 0
\end{tikzcd}.
\]
This is only possible if $M=0$ or $M=\nabla ' \B{i,j}$. Since $\iota$ is a proper inclusion, then $M=0$. This proves that $[\Ker{\pi}: L' _{i,1}]=0$. Thus $[\nabla ' \B{i,j}: L'_{i,1}] =1$ for all $(i,j) \in \Phi$. Consider now the module $N:=\nabla '\B{i,j}/L_{i,j}'$, which lies in $\mathcal{F}\B{\nabla'}$ as this category is closed under quotients. By what we have seen previously, $N$ has composition factors $L'_{i,1}, \ldots, L'_{i,j-1}$, with $L'_{i,1}$ having multiplicity one in $N$. The only possibility is that $N$ is isomorphic to $\nabla ' \B{i,j-1}$, that is $\nabla '\B{i,j}/L_{i,j}' \cong \nabla '\B{i,j-1}$.
\end{proof}
\begin{rem}
The proof of part (3) in Theorem \ref{thm:lastonebynow} can be simplified if the underlying field $K$ is algebraically closed.
\end{rem}
\begin{cor}
If $\B{B, \Phi, \sqsubseteq}$ is an ultra strongly quasihereditary algebra, then the algebra $\B{\mathcal{R}\B{B}\op, \Phi, \sqsubseteq \op}$ is also ultra strongly quasihereditary.
\end{cor}
\begin{proof}
By Theorem \ref{thm:lastonebynow}, it is clear that the quasihereditary algebra $\B{\mathcal{R}\B{B}, \Phi, \sqsubseteq \op}$ satisfies axioms (A1*) and (A2*).
\end{proof}

\section{The ADR algebra of a certain Brauer tree algebra}
\label{sec:last}
Brauer tree algebras are a class of algebras of finite representation type. They include all blocks of group algebras of finite type, and also all blocks of type A Hecke algebras of finite type (\cite{jost}). In this section we determine the quiver presentation of the ADR algebra $R_A$ of $A$, when $A$ is the Brauer tree algebra $KQ/I$, with $K$ an arbitrary field, $Q$ the quiver
\[
\begin{tikzcd}[ampersand replacement=\&]
\overset{1}{\circ} \arrow[bend left]{r}{\alpha_1}\& \overset{2}{\circ}\arrow[bend left]{l}{\beta_1} \arrow[bend left]{r}{\alpha_2}\&  \cdots\arrow[bend left]{l}{\beta_2} \arrow[bend left]{r}{\alpha_{n-2}}  \& \overset{n-1}{\circ} \arrow[bend left]{l}{\beta_{n-2}} \arrow[bend left]{r}{\alpha_{n-1}} \&
\overset{n}{\circ}\arrow[bend left]{l}{\beta_{n-1}}
\end{tikzcd}
\]
and $I$ the admissible ideal of $KQ$ generated by the relations
\[
\alpha_{i+1}\alpha_i,\,\, \beta_i \beta_{i+1},\,\, \alpha_i \beta_{i} - \beta_{i+1}\alpha_{i+1}, \,\, i=1 , \ldots, n-2.
\]

The Brauer tree algebra $A$ plays an important role in the representation theory of the symmetric group. Indeed, let $\Sigma_m$ be the symmetric group on $m$ letters. If $K$ is a field of prime characteristic $p$, then any non-simple block of $K \Sigma_m$ of finite type is Morita equivalent to the principal block of $K \Sigma_p$. Consider the algebra $A$ defined above, with $K$ a field of prime characteristic $p$ and with $n=p-1$. In this case $A$ is a basic algebra of the principal block of $K \Sigma_p$. Moreover, the vertex $i$ in the quiver of $A$ may be thought as corresponding to the simple $K \Sigma_p$-module labelled by the (hook) partition $(p+1-i,1^{i-1})$ of $p$. We refer to \cite{Scopes} for further details.

Since $I$ is generated by monomial relations and by commutative relations between paths of the same length, the projective indecomposable $A$-modules may be represented by graphs in the following way
\[
\begin{tikzcd}[ampersand replacement=\&, row sep =tiny, column sep = tiny]
1 \arrow[dash]{d} \\
2 \arrow[dash]{d} \\
1
\end{tikzcd},\,\,
\begin{tikzcd}[ampersand replacement=\&, row sep =tiny, column sep = tiny]
n \arrow[dash]{d} \\
n-1 \arrow[dash]{d} \\
n
\end{tikzcd}, \,\,
\begin{tikzcd}[ampersand replacement=\&, row sep =tiny, column sep = tiny]
\& i \arrow[dash]{dr} \arrow[dash]{dl} \& \\
i-1 \arrow[dash]{dr}\& \& i+1 \arrow[dash]{dl} \\
\& i  \&
\end{tikzcd}, \,\, i=2, \ldots, n-1.
\]
Denote the projective $A$-module corresponding to vertex $i$ by $P_i$.

By Section \ref{sec:sect2}, the $R_A$-modules $P_{i,1}=\Delta\B{i,1}$ are uniserial, with Loewy length 3, and with composition factors $L_{i,1}$, $L_{i,2}$, and $L_{i,3}$, ordered from top to socle. Furthermore, these projectives determine all the standard $R_A$-modules. Consider now (for $2 \leq i \leq n-1$) the short exact sequence
\[
\begin{tikzcd}[ampersand replacement=\&]
0 \arrow{r} \& L_{i+1} \oplus L_{i-1} \arrow{r} \& P_i/ \RAD{2}{P_i} \arrow{r}{\pi} \& L_i \arrow{r} \& 0 
\end{tikzcd},
\]
and apply $\Hom{A}{G}{-}$ to it. We get the exact sequence
\[
\begin{tikzcd}[ampersand replacement=\&]
0 \arrow{r} \& \Delta\B{i+1,1} \oplus \Delta\B{i-1,1} \arrow{r} \& P_{i,2} \arrow{r}{{\pi}_*} \& \Delta\B{i,1}
\end{tikzcd},
\]
and as ${\pi}_*\neq 0$, we must have $\Ima{\pi_*}=\Delta\B{i,2}$, since $\Delta\B{i,2}$ is the unique submodule of $\Delta\B{i,1}$ whose top is $L_{i,2}$. Note that this is exactly what Propositions \ref{prop:smalo} and \ref{prop:standard} are telling us. Similarly, we get
\[
\begin{tikzcd}[ampersand replacement=\&]
0 \arrow{r} \& \Hom{A}{G}{\Rad{P_i}} \arrow{r} \& P_{i,3} \arrow{r} \& \Delta\B{i,3} \arrow{r} \& 0
\end{tikzcd},
\]
and as $\Delta\B{i,3}=L_{i,3}$, it follows that $\Hom{A}{G}{\Rad{P_i}}=\Rad{P_{i,3}}$.

We wish to obtain a quiver presentation $KQ'/I'$ for $R_A$. As before, denote by $(i,j)$ the vertex of $Q'$ corresponding to the simple $R_A$-module $L_{i,j}$.
\begin{prop}
The algebra $R_A$ is isomorphic to $KQ'/I'$, with $Q'$ the quiver
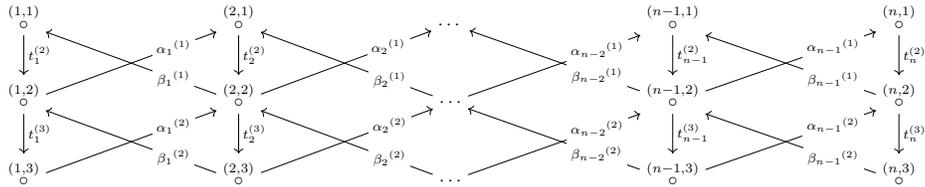
\[
\begin{tikzpicture}[baseline= (a).base]
\node[scale=.71] (a) at (0,0){
\begin{tikzcd}[ampersand replacement=\&, column sep =32mm]
\overset{(1,1)}{\circ} \arrow{d}{t_{1}^{(2)}} \& \overset{(2,1)}{\circ} \arrow{d}{t_{2}^{(2)}} \& \cdots \& \overset{(n-1,1)}{\circ} \arrow{d}{t_{n-1}^{(2)}} \& \overset{(n,1)}{\circ} \arrow{d}{t_{n}^{(2)}} \\
\overset{(1,2)}{\circ} \arrow{d}{t_{1}^{(3)}} \arrow{ur}[description, near end]{{\alpha_1}^{(1)}} \& \overset{(2,2)}{\circ} \arrow{d}{t_{2}^{(3)}} \arrow{ur}[description, near end]{{\alpha_2}^{(1)}} \arrow{ul}[description, near start]{{\beta_1}^{(1)}} \&  \cdots \arrow{ul}[description, near start]{{\beta_2}^{(1)}} \arrow{ur}[description, near end]{{\alpha_{n-2}}^{(1)}} \& \overset{(n-1,2)}{\circ} \arrow{d}{t_{n-1}^{(3)}} \arrow{ur}[description, near end]{{\alpha_{n-1}}^{(1)}} \arrow{ul}[description, near start]{{\beta_{n-2}}^{(1)}} \& \overset{(n,2)}{\circ} \arrow{d}{t_{n}^{(3)}} \arrow{ul}[description, near start]{{\beta_{n-1}}^{(1)}} \\
\overset{(1,3)}{\circ} \arrow{ur}[description, near end]{{\alpha_1}^{(2)}} \& \overset{(2,3)}{\circ} \arrow{ur}[description, near end]{{\alpha_2}^{(2)}} \arrow{ul}[description, near start]{{\beta_1}^{(2)}}  \& \cdots \arrow{ul}[description, near start]{{\beta_2}^{(2)}} \arrow{ur}[description, near end]{{\alpha_{n-2}}^{(2)}} \& \overset{(n-1,3)}{\circ} \arrow{ur}[description, near end]{{\alpha_{n-1}}^{(2)}} \arrow{ul}[description, near start]{{\beta_{n-2}}^{(2)}} \& \overset{(n,3)}{\circ} \arrow{ul}[description, near start]{{\beta_{n-1}}^{(2)}}
\end{tikzcd}
};
\end{tikzpicture}
\]
and $I'$ the admissible ideal generated by the relations
\begin{gather*}
\alpha_{i}^{(1)} t_i^{(2)}, \,\, \beta_{i}^{(1)} t_{i+1}^{(2)}, \,\, \alpha_i^{(2)} t_i^{(3)} - t_{i+1}^{(2)}\alpha_i^{(1)}, \,\, \beta_i^{(2)}t_{i+1}^{(3)}-t_i^{(2)} \beta_i^{(1)}, \,\, i=1, \ldots, n-1,
\\
{\alpha_{i+1}}^{(1)}{\alpha_i}^{(2)},\,\, {\beta_i}^{(1)} {\beta_{i+1}}^{(2)},\,\, {\alpha_i}^{(1)} {\beta_{i}}^{(2)} -{\beta_{i+1}}^{(1)}{\alpha_{i+1}}^{(2)}, \,\, i=1 , \ldots, n-2. 
\end{gather*}
\end{prop}
\begin{proof}
The vertical arrows in the quiver above correspond to the structure of the uniserial projectives $P_{i,1}$. In fact, going back to \cite{smalo}, one sees that the arrows
\[
\begin{tikzcd}[ampersand replacement=\&]
\overset{(i,j-1)}{\circ} \arrow{r}{t_i^{(j)}} \& \overset{(i,j)}{\circ}
\end{tikzcd}
\]
correspond to the canonical epics
\[
\begin{tikzcd}[ampersand replacement=\&]
P_i/\RAD{j}{P_i} \arrow[two heads]{r} \& P_i/ \RAD{j-1}{P_i}
\end{tikzcd}
\]
in $\Mod{A}$. Let $Q'$ be the ordinary quiver of $R_A$. Note that there must be exactly one arrow coming out of the vertices $(i,1)$ of $Q'$. Consider now the vertices $(i,3)$ of $Q'$. Because $P_i$ has Loewy length 3, it follows that
\[
\Rad{P_{i,3}}=\Hom{A}{G}{\Rad{P_i}}.
\]
It is not difficult to show directly that $\Rad{P_{i,3}}$ has top $L_{i-1,2} \oplus L_{i+1,2}$, $2 \leq i \leq n-1$. This also follows from Theorem A in \cite[Chapter 4]{RingeldualofR}. Consequently, there are exactly two arrows with source $(i,3)$ in $Q'$ (for $2 \leq i \leq n-1$), and they must be as depicted in the quiver above. Finally, let us analyse the vertices $(i,2)$ of $Q'$. By the structure of the modules $\Delta\B{i,2}$, there cannot exist arrows from $(i,2)$ to a vertex $(j,2)$. For the same reason, there cannot exist arrows from $(i,2)$ to $(j,3)$, apart from the arrow $t_i^{(3)}$ already mentioned. So any other arrow in $Q'$ having source $(i,2)$ (if any) must have sink $(j,1)$. That is, it must correspond to a map
\[
\begin{tikzcd}[ampersand replacement=\&]
L_j \arrow[hook]{r} \& P_i/ \RAD{2}{P_i}
\end{tikzcd}
\]
in $\Mod{A}$. Conversely, any monic as the one above must correspond to an arrow from $(i,2)$ to $(j,1)$ in $Q'$ because, by what we have seen so far, there cannot exist alternative paths from $(i,2)$ to $(j,1)$ in $Q'$. As a consequence, there must be two more arrows with source $(i,2)$ (if $2 \leq i \leq n-1$), namely
\[
\begin{tikzcd}[ampersand replacement=\&]
\overset{(i,2)}{\circ} \arrow{r}{\beta_{i-1}^{(1)}} \& \overset{(i-1,1)}{\circ}
\end{tikzcd}, \,\,\,
\begin{tikzcd}[ampersand replacement=\&]
\overset{(i,2)}{\circ} \arrow{r}{\alpha_{i}^{(1)}} \& \overset{(i+1,1)}{\circ}
\end{tikzcd} 
\]
This proves that $Q'$ coincides with the quiver in the statement of the proposition.

We have that $R_A \cong KQ'/I'$, for a certain admissible ideal $I'$. By the structure of $P_{i,1}$ the paths $\alpha_i^{(1)}t_i^{(2)}$, $\beta_i^{(1)}t_{i+1}^{(2)}$ must be zero modulo $I'$. Besides, $\alpha_i^{(2)} t_i^{(3)} - t_{i+1}^{(2)} \alpha_i^{(1)}$ must also be zero modulo $I'$ as the underlying diagram
\[
\begin{tikzcd}[ampersand replacement=\&]
P_{i+1}/ \RAD{2}{P_{i+1}} \arrow[two heads]{d}\arrow{r}{\neq 0} \& P_i/ \RAD{3}{P_i} \arrow[two heads]{d} \\
L_{i+1} \arrow[hook]{r} \& P_i/ \RAD{2}{P_i}
\end{tikzcd} 
\]
commutes. Similarly, it follows that $\beta_i^{(2)}t_{i+1}^{(3)}-t_i^{(2)} \beta_i^{(1)}$ must be zero modulo $I'$. In a similar fashion one checks that the remaining relations in the statement of the proposition are zero modulo $I'$. Let $\hat{I}$ be the ideal of $KQ'$ generated by the relations indicated in the statement of the proposition. There is an epic from $KQ' / \hat{I}$ to $R_A$. It is not difficult to check that $R_A$ has dimension $19 n -10$ as a $K$-vector space. It is also easy to prove by induction on $n$ that the dimension of $KQ' / \hat{I}$ is given by the same expression.
\end{proof}

We conclude with some remarks about the algebra $R_A=KQ'/ I'$.
\begin{rem}
Note that the arrows ${\beta}_{i-1}^{(1)}$, ${\alpha}_{i}^{(1)}$ in $Q'$ correspond to irreducible maps in $\Mod{A}$. Let $M$ be a module in $\Mod{A}$. It is clear that any irreducible map $f:X \longrightarrow Y$, with $X$, $Y$ in $\Add{M}$, gives rise to a morphism $f_*=\Hom{A}{M}{f}$ between projectives in $\Mod{(\END{A}{M}\op)}$, satisfying $\Ima{f_*} \subseteq \Rad{\Hom{A}{M}{Y}}$, $\Ima{f_*} \not\subseteq \RAD{2}{\Hom{A}{M}{Y}}$.
\end{rem}
\begin{rem}
Let $A$ be as before. By Theorem $10.3$ in \cite{MR0349747}, $\Gl{R_A} \leq 3$. Proposition 2 in \cite{smalo1978} implies that $\Gl{R_A} \neq 2$. Hence $\Gl{R_A}=3$. Moreover, according to Theorem B in \cite[Chapter 4]{RingeldualofR}, the Ringel dual of $R_A$ is isomorphic to $\B{R_A} \op$ for every Brauer tree algebra $A$.
\end{rem}

\bibliographystyle{amsplain}
\bibliography{QHAlg}

\end{document}